\documentclass{wiki}
\usepackage{filecontents}

\title{On existence of Borel flow for ordinary differential equation with a non-smooth vector field}
\author{Nikolay A. Gusev}

\date{%
    \small (Steklov Mathematical Institute of Russian Academy of Sciences, RUDN~University~of~Russia)\\[2ex]%
    \today
}

\begin{document}

\maketitle

\begin{abstract}
For smooth vector fields the classical method of characteristics provides a link between the ordinary differential equation and the corresponding continuity equation (or transport equation). We study an analog of this connection for merely bounded Borel vector fields.
In particular we show that, given a non-negative Borel measure $\bar \mu$ on $\R^d$, existence of $\bar \mu$-measurable flow of a bounded Borel vector field is equivalent to existence of a measure-valued solution to the corresponding continuity equation with the initial data $\bar \mu$.
\end{abstract}


\vspace{3pt}


\section{Introduction and statement of the main result}

Let $b\colon \I\times \R^d \to \R^d$ be a bounded Borel vector field, where $\I=[0,T]$, $T>0$.

Consider the Cauchy problems
\begin{equation}\label{ode-classic}
\begin{cases}
\d_t \Fhi_t(x) = b(t,\Fhi_t(x)), \quad t\in(0,T) \\ 
\Fhi_0(x) = x
\end{cases}
\end{equation}
and
\begin{equation}\label{continuity-classic}
\begin{cases}
\d_t u(t,x) + \div_x (u(t,x) b(t,x)) = 0, \quad (t,x)\in (0,T)\times \R^d, \\
u(t,x)|_{t=0} = \bar u(x)
\end{cases}
\end{equation}
The function $\Phi_t(\cdot)$ is called the \emph{(phase) flow} of $b$.

If $b$ and $\bar u$ are smooth then the classical method of characteristics allows to represent the solution of \eqref{continuity-classic} using the flow of $b$ (and vice versa):
\begin{equation*}
u(t,\cdot) = \frac{\bar u}{\det \nabla_x \Fhi_t} \circ \Phi_t,
\end{equation*}
i.e. $u(t,x) = \bar u(y)/\det(\nabla \Fhi_t)(y)|_{y = \Fhi_t^{-1}(x)}$.

In this work we study the relation between the Cauchy problems \eqref{ode-classic} and \eqref{continuity-classic} under significantly weaker assumptions. Namely, we will assume that $b$ is Borel and bounded. 
Nonsmooth vector fields arise in many areas of mathematical physics, in particular, in hyperbolic conservation laws, fluid mechanics and kinetic theory \cite{AmbrosioBV,DiPernaLions}.

For convenience of the reader let us recall the definitions of solutions of \eqref{ode-classic}
and \eqref{continuity-classic} for non-smooth vector fields.
Let 
\begin{equation}\label{e-Gamma-def}
\Gamma:=\Lip_{\|b\|_\infty}(I;\R^d),
\end{equation}
where $\Lip_M(I; \R^d)$ denotes the set of Lipschitz functions $f\colon I\to \R^d$ satisfying
$|f(t) - f(t')|\le M |t-t'|$ for all $t,t' \in I$. We consider $\Gamma$ as a metric space endowed
with the metric of $C(I;\R^d)$.

\begin{definition}
$\gamma\in C(I;\R^d)$ is called an integral curve of $b$ if
$$
\gamma(t) = \gamma(0) + \int_0^t b(s,\gamma(s)) \, ds
$$
for all $t\in I$. An integral curve $\gamma$ of $b$ is called a solution of the Cauchy problem
\begin{equation}\label{ode-weak}
\begin{cases}
\d_t \gamma = b (t,\gamma), \quad t\in(0,T) \\ 
\gamma|_{t=0} = x
\end{cases}
\end{equation}
if $\gamma(0)=x$.
\end{definition}

Let $\Gamma_b$ denote the set of the integral curves of $b$. Clearly $\Gamma_b \subset \Gamma$.

\begin{definition}
A family $\{\mu_t\}_{t\in \I}$ of locally finite (possibly signed) Borel measures on $\R^d$ is called \emph{(Lebesgue) measurable} if
\begin{enumerate}
\item for any bounded Borel set $B\subset \R^d$ the map $t\mapsto |\mu_t|(B)$ is Lebesgue-measurable;
\item for any compact $K\subset \R^d$ the map $t\mapsto |\mu_t|(K)$ belongs to $L^1(\I)$.
\end{enumerate}
Here $|\mu|$ denotes the total variation of a Borel measure $\mu$.
\end{definition}

\begin{definition}
Suppose that $\bar \mu$ is a locally finite measure on $\R^d$. A measurable family $\{\mu_t\}_{t\in \I}$ is called a \emph{(measure-valued) solution of}
\begin{equation}\label{continuity-weak}
\d_t \mu_t + \div(b\mu_t) = 0, \quad \mu_t|_{t=0} = \bar \mu
\end{equation}
if for any $\fhi\in C^1(\R^{d+1})$ such that $\fhi=0$ on $\I\times (\R^d \setminus B_R(0))$ for some $R>0$ the equality
$$
\int_{\R^d} \fhi(t,x) d\mu_t(x) = \int_{\R^d} \fhi(0,x) d\bar\mu(x) + \int_0^t \int_{\R^d} (\d_t \fhi + b\cdot \nabla \fhi) \, d\mu_t \, dt
$$
holds for a.e. $t\in \I$.
\end{definition}

Let us denote with $\L^d$ the Lebesgue measure on $\R^d$. It is easy to see that if $u$ is a classical solution of \eqref{continuity-classic} then $\mu_t := u(t,\cdot) \L^d$ is a measure-valued solution of \eqref{continuity-weak}.
However measure-valued solutions in general are not absolutely continuous with respect to Lebesgue measure.
For example, if $\gamma$ solves \eqref{ode-weak} then $\mu_t:= \delta_{\gamma(t)}$ solves \eqref{continuity-weak}.

The starting point of the present work is the following result proved by O.E. Zubelevich:

\begin{theorem}[see \cite{Zubelevich2012DAN,Zubelevich2012FE}]\label{existence-of-borel-flow-continuous}
If $b\in C(\I\times \R^d; \R^d)$ then there exists a Borel function $F\colon \R^d \to \Gamma$ such that $F(x)$ is a solution of (1) for any $x\in \R^d$.
\end{theorem}

Zubelevich has also proved that under the assumptions of Theorem~\ref{existence-of-borel-flow-continuous} for any locally finite Borel measure $\bar \mu$ on $\R^d$ the problem \eqref{continuity-weak} has a measure-valued solution with the initial data $\bar \mu$.
In fact the following rather well-known independent statement holds:

\begin{theorem}\label{existence-of-flow-implies-existence-of-mvs}
Suppose that $\bar \mu$ is a locally finite Borel measure on $\R^d$ and there exists a Borel function $F\colon \R^d \to \Gamma$ such that $\forall x\in \R^d$ the curve $F(x)$ solves \eqref{ode-weak}. Then there exists a measure-valued solution of \eqref{continuity-weak} with the initial condition $\bar \mu$.
\end{theorem}

Our main result is an inverse version of Theorem~\ref{existence-of-flow-implies-existence-of-mvs}:

\begin{theorem}\label{existence-of-mvs-implies-existence-of-flow}
Suppose that $\{\mu_t\}_{t\in \I}$ is a measurable family of \emph{non-negative} locally finite Borel measures on $\R^d$ which solve \eqref{continuity-weak} with some initial condition $\bar \mu \ge 0$. Then there exists a Borel function $F\colon \R^d \to \Gamma$ such that $F(x)$ is a solution of \eqref{ode-weak} for $\bar \mu$-a.e. $x\in \R^d$.
\end{theorem}

Note that existence of signed measure-valued solution in general might provide no useful information about the integral curves. In particular one can take $b(x) = -1$, $x \ge 0$ and $b(x)=1$, $x<0$ and $\bar \mu = 1_{(0,+\infty)}\L^1 - 1_{(-\infty, 0)}\L^1$, where
$1_E$ denotes the characteristic function of $E\subset \R^d$:
\begin{equation*}
1_E(x) = \begin{cases}
1, x\in E, \\
0, x\notin E.
\end{cases}
\end{equation*}
Indeed, $(b \bar \mu)'=0$ but there are no integral curves starting from $x\in [0,T)$ (after touching zero they are not defined). Therefore assumption that $\mu_t\ge 0$ is crucial.
Moreover, Theorem~\ref{existence-of-mvs-implies-existence-of-flow} shows that for this vector field and for any positive measure $\bar \mu$ on $[0,T)$ the problem \eqref{continuity-weak} has no measure-valued solutions.

In the example above it is possible to modify vector field by redefining the value of $b$ at $x=0$. Namely, for $\tilde b = 1_{(-\infty,0)} - 1_{(0,+\infty)}$ the Cauchy problem \eqref{continuity-weak} has a weak solution, in particular, for any finite positive measure $\bar \mu$ on $[0,T)$. However the encountered ``pathology'' cannot always be removed by redefining the vector field on a Lebesgue-negligible subset. For example (see \cite{PS2012flows}), if $K\subset [0,1]$ is a closed totally disconnected set of positive Lebesgue measure
and $b = 1_K$ then solutions of \eqref{ode-weak} do not exist for $x\in K$.

In view of Theorems~\ref{existence-of-flow-implies-existence-of-mvs} and~\ref{existence-of-mvs-implies-existence-of-flow} existence of Borel flow is, in a certain sense, equivalent to existence of a nonegative measure-valued solution of \eqref{continuity-weak}:

\begin{corollary}
Let $\bar \mu \ge 0$ be a locally finite Borel measure on $\R^d$.
Then the following statements are equivalent
\begin{enumerate}
\item there exists a Borel function $F\colon \R^d \to \Gamma$ such that for $\bar \mu$-a.e. $x\in \R^d$ the curve $F(x)$ solves (1);
\item the problem (2) with the initial data $\bar \mu$ has a non-negative measure-valued solution.
\end{enumerate}
\end{corollary}

\section{Proof of the main result}

Given metric spaces $X$ and $Y$ and Borel function $f\colon X\to Y$ let $f_\# \mu$ denote the push-forward (or image) of a measure $\mu$ on $X$ under $f$. I.e. $(f_\#\mu)(B) := \mu(f^{-1}(B))$ for any Borel set $B\subset Y$.
In this notation $\int \fhi(f(x))\, d\mu(x) = \int \fhi(y) d\, (f_\# \mu)(y)$ for any bounded Borel function $\fhi\colon Y\to \R$.

Let us denote by $e_t$ the map $e_t\colon \Gamma \to \R^d$ defined by $e_t(\gamma):=\gamma(t)$.
We also denote by $e$ the map $e\colon \I\times \Gamma \to \I\times \R^d$ defined by $e(t,\gamma) := (t, \gamma(t))$.

It is easy to see that Theorem~\ref{existence-of-flow-implies-existence-of-mvs} can be proved
by checking that $$\mu_t := (\Phi_t)_\# \bar \mu$$ solves \eqref{continuity-weak}, where $\Phi_t(x) := e_t(F(x))$.

The proofs of Theorem~\ref{existence-of-mvs-implies-existence-of-flow} and Theorem~\ref{existence-of-borel-flow-continuous} rely on the following corollary of the measurable selection theorem:

\begin{theorem}[see e.g. \cite{bogachev2007measure}, Th. 6.9.7]\label{measurable-selection}
 Let $X$ be a compact metric space, $Y$ be a Hausdorff topological space and $f\colon X \to Y$ be a continuous mapping. Then there exists a Borel set $B\subset X$ such that $f(B) = f(X)$, $f$ is injective on $B$ and the mapping $f^{-1}\colon f(X) \to B$ is Borel.
\end{theorem}

Before proving Theorem~\ref{existence-of-mvs-implies-existence-of-flow}, let us briefly mention the main idea of the proof of Theorem~\ref{existence-of-borel-flow-continuous}.
If $b$ is continuous (or, more generally, $L^1(I; C(\R^d))$) then Arzel\`{a}--Ascoli theorem imples that $\Gamma_b$
is compact. Therefore, using Theorem~\ref{measurable-selection} with $X=\Gamma_b$ and $f=e_0$, one can
select a Borel flow $F$ of $b$.

When $b$ is not continuous it is no longer evident that $\Gamma_b$ is compact.
However, if there exists a non-negative regular Borel measure $\eta$ on $\Gamma$ and $\eta$ is concentrated on $\Gamma_b$
(i.e. $\eta(\Gamma \setminus \Gamma_b) = 0$), then $\Gamma_b$ is $\sigma$-compact (up to an $\eta$-negligible subset).

Existence of a suitable measure $\eta$ follows from existence of a non-negative measure-valued solution of \eqref{continuity-weak} in view of the following Superposition Principle (see \cite{AmbrosioBV} and Theorem 8.2.1 in \cite{AGS2008}):

\begin{theorem}[Superposition Principle]
If $\{\mu_t\}_{t\in \I}$ is a measurable family of non-negative locally finite Borel measures on $\R^d$ which solves \eqref{continuity-weak} then there exists a non-negative locally finite Borel measure $\eta$ on $\Gamma$ such that
$\eta$-a.e. $\gamma \in \Gamma$ is an integral curve of $b$ and
\begin{equation}\label{representation-of-mvs}
\mu_t= (e_t)_\# \eta
\end{equation}
holds for a.e. $t\in (0,T)$.
\end{theorem}

\goodbreak

\begin{proof}[Proof of Theorem~\ref{existence-of-mvs-implies-existence-of-flow}]
Let us prove that $\Gamma_b$ is a $\eta$-measurable set (the set $\Gamma_b$ is in fact Borel (see \cite{Bernard2008}, Proposition 2).

Let $t_k$, $k\in \N$, denote the rational points in $\I$. Then
\begin{equation*}
\Gamma_b = \bigcap_{k\in \N} G_k^{-1}(0)
\end{equation*}
where $G_k(\gamma) = \gamma(t_k) - \gamma(0) - \int_0^{t_k} b(s,\gamma(s)) \, ds$.
Hence it is sufficient to prove that for any $t\in \I$ the map $$G(\gamma):=\gamma(t) - \gamma(0) - \int_0^t b(s,\gamma(s))\, ds$$
is $\eta$-measurable.

Consider the measure $\L^1 \times \eta$ on $[0,t]\times \Gamma$.
The function $(s,\gamma) \mapsto b(s,\gamma(s))$ is Borel as the composition $b\circ e$, where $b\colon [0,t]\times \R^d \to \R^d$ is Borel and $e \colon [0,t]\times \Gamma \to [0,t] \times \R^d$ is continuous. Hence $b\circ e \in L^1(\L^1 \times \eta)$ and therefore by Fubini's theorem the function
\begin{equation*}
\gamma \mapsto \int_0^t b(s,\gamma(s))\, ds
\end{equation*}
is $\eta$-measurable. The functions $\gamma \mapsto \gamma(t)$ and $\gamma \mapsto \gamma(0)$ are continuos hence we have proved that $G$ is $\eta$-measurable.

By \eqref{e-Gamma-def} and Arzel\`a-Ascoli theorem the sets $\Gamma^m=\{\gamma \in \Gamma\colon \; |\gamma(0)| \le m\}$ are compacts for all $m\in \mathbb N$. Since $\Gamma = \cup_{m\in \mathbb N} \Gamma^m$ the space $\Gamma$ is locally compact. Moreover, $\eta$ is $\sigma$-finite,
since
\begin{equation*}
\eta(\Gamma^m) = \int 1_{\bar B(0,m)}(\gamma(0)) \, d\eta(\gamma) = \int 1_{\bar B(0,m)}(x) \, d\bar \mu(x) < \infty,
\end{equation*}
where $\bar B(x,r) = \{y\in \R^d \colon \; |x-y|\le r\}$.
Finally, $\Gamma$ is separable (being a subset of separable space $C(I; \R^d)$). 
Hence $\eta$ is inner regular
(see e.g. \cite{AFP}, Prop. 1.43). 

By inner regularity of $\eta$ we can find a sequence of compacts $K_n\subset \Gamma_b$ such that $K_n \subset K_{n+1}$ and 
\begin{equation*}
\eta(\Gamma_b \setminus \cup_{n=1}^\infty K_n) = 0.
\end{equation*}

Since $e_0 \colon \Gamma \to \R^d$ is continuous, for each $n\in \N$ by Theorem~\ref{measurable-selection} there exists a Borel function $f_n\colon e_0(K_n) \to K_n$ such that $e_0(f_n(x)) = x$ for any $x\in e_0(K_n)$. We define
\begin{equation*}
F(x):=\begin{cases}
f_n(x), \quad x \in e_0(K_n) \setminus \cup_{j=1}^{n-1} e_0(K_j), \\
0, \quad x \in N,
\end{cases}
\end{equation*}
where $N:= \R^d \setminus \cup_{n=1}^\infty e_0(K_n)$. Clearly the function $F$ is a pointwise limit of Borel functions, hence $F$ is Borel.

It remains to compute that
\begin{equation*}
\bar \mu(N) = \int 1_N(\gamma(0))\, d\eta(\gamma) = \eta(\Gamma_b \setminus \cup_{n=1}^\infty K_n) = 0. \qedhere
\end{equation*}
\end{proof}

\section{Acknowledgements}
This work was supported by the Ministry of Education and Science of the Russian Federation (the Agreement number \No 02.a03.21.0008).
The author would like to thank P. Bonicatto, V.P. Burskii and V.V. Zharinov for interesting disussions and their valuable remarks.

\bibliographystyle{unsrt}
\bibliography{bibliography}

\bigskip

\hspace{18em}\vbox{\footnotesize%
\noindent\textit{Nikolay A. Gusev}\\
Department of Mathematical Physics,\\
Steklov Mathematical Institute\\
of Russian Academy of Sciences,\\
8 Gubkina St, Moscow, 119991;\\
Department of Applied Mathematics,
\\RUDN University,\\
6 Miklukho-Maklay St, Moscow, 117198;\\
{E-mail address}: \href{mailto:n.a.gusev@gmail.com}{\texttt{n.a.gusev@gmail.com}}}

\end{document}